\documentclass[reqno]{amsart}

\usepackage{ upgreek }
\usepackage{amssymb}
\usepackage{amsmath}
\usepackage{amsthm}
\usepackage{epic}
\usepackage{accents}
\usepackage{enumitem} 
\usepackage[mathscr]{euscript}
\usepackage{amsfonts}
\usepackage[utf8]{inputenc}
\usepackage{fancyhdr}
\usepackage{mathtools}
\usepackage{comment}

\newtheorem{theorem}{Theorem}[section]
\newtheorem{lemma}[theorem]{Lemma}
\newtheorem*{theorem*}{Theorem}
\newtheorem*{obs*}{Observation}
\newtheorem{proposition}[theorem]{Proposition}
\newtheorem{conj}[theorem]{Conjecture}

\newtheorem{corollary}[theorem]{Corollary}
\theoremstyle{definition}
\newtheorem {remark}[theorem]{Remark}
\newtheorem {fact}[theorem]{Fact}
\newtheorem {definition}[theorem]{Definition}

\usepackage{xcolor}

\def \C{\mathbb{C}}
\def \Z{\mathbb{Z}}
\def \Q{\mathbb{Q}}
\def \N{\mathbb{N}}
\def \R{\mathbb{R}}
\def \H{\mathcal{H}}

\def \ge{\geqslant}
\def \le{\leqslant}

\newcommand{\Ranexp}{\mathbb{R}_{\textnormal{an},\exp}}

\newcommand{\Rsubexp}{\mathcal{R}_{\textnormal{subexp}}}
\newcommand{\hg}{\hat{g}}

\copyrightinfo{}{}
\issueinfo{}
  {}
  {}
  {}

\PII{}
\def\@serieslogo{\@empty}

\title{Integer-valued o-minimal functions}

\author[Bhardwaj]{Neer Bhardwaj}
\address{Weizmann Institute of Science, Israel}
\email{nbhardwaj@msri.org}

\author[McCulloch]{Raymond McCulloch$^1$}
\address{University of Manchester, UK}
\email{raymond.mcculloch@manchester.ac.uk}
\thanks{$^1$RM thanks the Heilbronn Institute of Mathematical Research for their support.}

\author[Ramachandran]{Nandagopal Ramachandran}
\address{University of California San Diego, USA}
\email{naramach@ucsd.edu}

\author[Woo]{Katharine Woo$^2$}
\address{Princeton University, USA}
\email{khwoo@princeton.edu}
\thanks{$^2$KW is supported by the National Science Foundation under Award No. DGE-2039656.}
\begin{document}

\begin{abstract}
We study $\Ranexp$-definable functions $f:\R\to \R$ that take integer values at all sufficiently large positive integers. If  $|f(x)|= O\big(2^{(1+10^{-5})x}\big)$, then we find polynomials $P_1, P_2$ such that $f(x)=P_1(x)+P_2(x)2^x$ for all sufficiently large $x$. Our result parallels classical theorems of P\'olya and Selberg for entire functions and generalizes Wilkie's classification for the case of $|f(x)|= O(C^x)$, for some  $C<2$. 


Let $k\in \N$ and $\gamma_k=\sum_{j=1}^{k} 1/j$. Extending Wilkie's theorem in a separate direction, we show that if $f$ is {\em $k$-concordant} and $|f(x)|= O(C^{x})$, for some $C<e^{\gamma_k}+1$, then $f$ must eventually be given by a polynomial. This is an analog of a result by Pila for entire functions.
\end{abstract}

\maketitle

\section{Introduction}

\noindent
In \cite{Wilkie}, Wilkie studies functions $f:\R\to \R$  definable in the o-minimal structure $\Ranexp$ with the property that $f(a)\in \Z$ for all sufficiently large positive integers $a$. If there is a constant $C<2$ such that $|f(x)|\le C^x$ for all sufficiently large  $x$, then \cite[Theorem 5.1]{Wilkie} shows that $f$ is eventually given by a polynomial. This parallels a special case of the following seminal theorem by P\'olya for entire functions \cite{Polya}. 

\begin{theorem}\label{P}
Let $g:\C\to \C$ be an entire function which satisfies $g(a)\in \Z$ for all positive integers $a$. Suppose there is $M\in \N$ such that 
$$|g(z)|= O(|z|^M 2^{|z|}) \qquad \text{as}\ \ |z|\to \infty.$$
Then there are polynomials $P_1, P_2$ such that $g(z)=P_1(z)+P_2(z)2^z$ for all $z\in \C$.
\end{theorem}

\noindent

\noindent
We show that a similar characterization holds for $\Ranexp$-definable functions under a slightly more lenient restriction on growth.  Our statement can be seen to correspond to Selberg's
extension of Theorem~\ref{P} in \cite{Selberg}, which covers all entire functions that have growth bounded by $e^{(\ln(2)+1/1500)|z|}$.

\noindent

\begin{theorem}\label{main}
Let $f:\R \to\R$ be definable in $\Ranexp$. Suppose $f(a)\in \Z$ for all sufficiently large positive integers $a$ and  that $$|f(x)|= O (2^{(1+10^{-5})x}) \qquad \text{as}\ \ x\to \infty.$$
Then there are polynomials $P_1, P_2 \in \Q[X]$ such that $f(x) = P_1(x) + P_2(x)2^x$ for all sufficiently large $x$. 
\end{theorem}

\noindent
The proof of this theorem broadly follows Wilkie's strategy for the aforementioned result from \cite{Wilkie}. This approach relies on the  important fact, established in \cite{Wilkie}, that an $\Ranexp$-definable function on the reals is {\em closely} approximated by a complex-analytic function  in an appropriate right half-plane. We use this ingredient without change, 
and our main technical contribution lies in the more extensive reworking of Langley's generalization of Theorem~\ref{P}, where the same characterization is achieved for functions
analytic in a right-half plane \cite{Langley}.

In fact, our adaptation can be used to obtain a minor improvement to the primary theorem of Langley; see Remark~\ref{bL}. This small upgrade would certainly not be a novel result, and indeed more comprehensive classifications are known for integer-valued complex-analytic functions, but via different methods (\cite{Buck, Yoshino}). At the end of \cite{Wilkie}, an extensive analog of these classifications is conjectured for $\Ranexp$-definable functions of exponential type, and we view Theorem~\ref{main} as a step towards this larger goal; see Conjecture~\ref{conj: Wilkie}.

\smallskip
\noindent
Our second main theorem is a direct analog of a result by Pila, which studies entire functions satisfying a strong interpolative property. For $k\in \N$, we say that a function $f:\R \to \R$ is {\em $k$-concordant} if for each tuple $m_0,\dots,m_{k}$  of sufficiently large positive integers, there is a polynomial $P\in\Z[X]$ such that $P(m_i)=f(m_i)$ for all $i=0,\ldots, k$. Set $\gamma_k\coloneq\sum_{j=1}^{k} 1/j$, and $\gamma_0=0$ by convention. 

\begin{theorem}\label{concord}
Let $f:\R\to \R$ be definable in $\Ranexp$. Suppose $f$ is $k$-concordant for some $k\in \N$, and that there is $C< e^{\gamma_{k}}+1$ such that $$|f(x)|\le C^x$$
for all large enough $x$. Then there is a polynomial $P\in \Q[X]$ such that $f(x)=P(x)$ for all sufficiently large $x$.
\end{theorem}

\noindent
Pila's theorem~\cite[Theorem 1.3]{Pila} shows that a similar statement holds for entire functions. For $k=0$, Theorem~\ref{concord} is exactly \cite[Theorem 5.1]{Wilkie}, and  for the case of $k=1$, our result corresponds to a theorem  by Perelli and Zannier \cite{PeZ} for entire functions; see Remark~\ref{PeZa}.

\smallskip
\noindent
The present paper fits into the larger theme of interactions of o-minimality with diophantine issues, and we view \cite[Corollary 2.2]{wilkie0} as the starting point for the particular trajectory pursued in this article. The results in \cite{wilkie0} were succeeded also by the celebrated Pila-Wilkie theorem \cite{PW}, a specific variant of which is used in fact by Jones, Thomas, and Wilkie in \cite{JTW}  to obtain a precursor to Theorem~\ref{main}. Moreover,  Jones and Qiu \cite{GQ} obtain similar results for $\Ranexp$-definable functions which are either {\em close} to being integer-valued or take integer values on a sufficiently dense subset of the natural numbers.

\subsection*{Notations and conventions.} Throughout we work with the conventions that $\mu, \nu, i, j,k \in \N\coloneq\{0,1,\ldots\}$, $ a, m, n, B, K, N\in \N^{\ge 1}$, $p$ is prime, $ \epsilon, c, C\in \R^{>0}$, and $r\in\R$. Any symbol with an added subscript will denote an object in the same domain as the plain symbol. So $B_1\in \N^{\ge 1}$, $p_0$ is a prime, $C_{\star}\in \R^{> 0}$, and so on. 

The variable $x$ varies always over the real numbers and $e$ denotes Euler's number. For any $r\in \R$, we set $\H(r)\coloneq\{z\in \C:\ \mathrm{Re}(z) >r\}$; so $\H(0)$ denotes the standard right-half plane. 

We will work with a difference operator $\Delta$ given as follows. For any $f:\R\to \R$, set $\Delta f(x)\coloneq f(x+1)-f(x)$. We let $\Delta^n$  denote the $n^{\mathrm{th}}$-iterate of the $\Delta$ operator, and by a straight-forward induction we have for all $x$ that 
\begin{equation}\label{eq: Delta^n}\Delta^n f(x)\ =\ \sum_{k=0}^n (-1)^{n-k} \frac{n!}{k! (n-k)!} f(x+k),\end{equation}
which can also be seen in \cite{Wilkie}. We shall also work with the $\Delta-1$ operator, given by $(\Delta-1)f(x)=f(x+1)-2f(x)$, and we consider iterates of this operator as well. If $f(x)=2^xh(x)$, then for all $x$ we have that
$(\Delta -1)^nf(x)\ =\ 2^{x+n}\Delta^nh(x)$.

\subsection*{Acknowledgements}  We are deeply grateful to the organisers of the Arizona Winter School 2023 for bringing us together, and making this project possible. We thank  Gabriel Dill and Jonathan Pila for introducing us to the problem which evolved into Theorem~\ref{concord}; and extend a very special thanks to  Gabrielle Scullard, who was part of our working group at the AWS. We express our gratitude also to Gal Binyamini, Gareth Jones, Yuval Salant, and Alex Wilkie for helpful discussions during the preparation of this article.

\section{Technical preliminaries}\label{tech}

\noindent
In this section, we record some important ingredients to be used in the proofs of Theorems~\ref{main} and \ref{concord} in \S\ref{IV} and \S\ref{CIV} respectively.

\medskip
\noindent
Suppose $f:\R \to \R$ be given by $f(x) = P_1(x) + P_2(x)2^x$, for some polynomials $P_1, P_2$. Then $\Delta^n(\Delta-1)^kf=0$ for any $n\ge \deg P_1 + 1$ and any $k \ge \deg P_2 + 1$. The following sharp converse will be crucial for us;  the statement may be viewed as an appropriate \cite[Lemma 5.2]{Wilkie} type analog of \cite[Lemma 3]{Langley}.

\begin{lemma}\label{inter}
Let a function $f:\N\to\R$ be given. Suppose there are $B, K$ such that for all $a\ge B$ we have that $$\Delta^{n-a}(\Delta-1)^{a} f(a)=0$$ for all $n$ satisfying $Ka\le n\le K(a+1)$. Then there are polynomials $P_1, P_2$ such that $f(a) = P_1(a)  + P_2(a)2^a$ for all $a\ge B$.
If moreover, $f(a)\in\Z$ for all $a\ge B$, then we get that $P_1, P_2\in \Q[X]$.
\end{lemma}

\begin{proof}
Let $f:\N\to \R$ be given, and suppose $B,K$ are  such that for all $a\ge B$, we have that $\Delta^{n-a}(\Delta - 1)^af(a) = 0$ for all $n$ satisfying $Ka\le n \le K(a + 1)$.
Let $P_1$ be a polynomial of degree at most $(K - 1)B-1$ and $P_2$ be a polynomial of degree at most $B-1$ such that
$$P_1(m) + P_2(m)\cdot2^m = f(m)$$
for all $B\le  m \le KB + B-1$. Let $h:\N\to \R$ be given by
$$h(m) = P_1(m) + P_2(m)2^m - f(m).$$
If $n -a \ge (K - 1)B$, then $\Delta^{n - a}P_1(m) = 0$ for all $m$. Similarly, $a \ge B$ implies $(\Delta - 1)^a\big(P_2(m)2^{m}\big) = 2^{m+a}\Delta^a P_2(m)= 0$ for all $m$. So we have for all $a \ge B$ that
$$\Delta^{n - a}(\Delta - 1)^a\big(P_1(a) + P_2(a)2^{a}\big) = 0$$
for all $n\ge Ka$. Then by the linearity of $\Delta$, the definition of $h$, and the starting assumption on $f$, we have for all $a\ge B$ that
$$\Delta^{n - a}(\Delta - 1)^ah(a) = 0$$
for all $n$ satisfying $Ka\le n \le K(a + 1)$.

We shall show that $h(a)=0$ for all $a\ge B$, which clearly proves the first assertion of the lemma. To achieve this goal, we prove by induction on $a$ that 
$$h(m) = 0\quad \text{ for all }\quad B\le m \le Ka +a -1.$$
We have the base case $a=B$ by the definition of $h$. Fix $a\ge B$ and assume the desired statement holds for $a$. It remains to prove the statement for $a+1$, i.e.
$$h(m)=0\quad \text{ for all } \quad Ka + a  \le m \le K(a+1) + a $$
For all $n$ satisfying $Ka\le n\le K(a+1)$, we have that $\Delta^{n - a}(\Delta - 1)^ah(a) = 0$, which gives that 
$$h(a + n) = \text{ a linear combination of }\ h(a), h(a + 1), \ldots , h(a + n - 1).$$
If $n=Ka$, then $a+j\le a+ Ka -1$ for all $j=1,\ldots, n-1$.
By the inductive assumption we get $h(a) = \ldots = h(a + Ka - 1) = 0$, and it follows that $h(a + Ka) = 0$. Proceeding in the same manner, we immediately conclude that $h(m)=0$ for all $m\le a + K(a+1)$, and the proof of the first part of the lemma is complete.

For the second assertion in the lemma, we need only observe that $f(a)\in \Z$ for all $a\ge B$ gives that the polynomials $P_1$ and $P_2$  above have rational coefficients.
\end{proof} 

\noindent
The following simple fact will help handle the approximation error induced from deploying Proposition~\ref{3.1}.
\begin{lemma}\label{error}
Let $K\ge 2$, and suppose $H: \R\to \R $ is such that $|H(x)|\le \exp(-Kx)$ for all sufficiently large $x$. Then for all sufficiently large $a$,   we have that 
$$|\Delta^{n-a}(\Delta-1)^a H(a)| < 1/2$$
for all $a \le n \le K(a+1)$.
\end{lemma}
\begin{proof} Let $B$ be such that $|H(x)|\le \exp(-Kx)$ for all $x\ge B$. Fix an $a\ge B$ and an $a \le n \le K(a+1)$. Then we have that
\begin{align*}
    |\Delta^{n-a}(\Delta-1)^a H(a)|\ &\le \ \sum_{j=0}^a  \frac{a!}{j! (a-j)!}\cdot |\Delta^{n-j}H(a)| \\
    & \le\ \sum_{j=0}^a  \frac{a!}{j! (a-j)!}\cdot 2^{n-j} \exp(-Ka)\\
    & \le\ 2^n\cdot(3/2)^a\cdot \exp(-Ka)\\
    & \le\ 2^K\cdot\big( 2^K\cdot (3/2)\cdot\exp(-K)\big)^a,
\end{align*}
using that $n\le K(a+1)$.
Since $K\ge 2$, we have that $(3/2)\cdot (2/e)^K<1$. Hence for all sufficiently large $a$,  we get that $|\Delta^{n-a}(\Delta-1)^a H(a)|<1/2$  for all $n \le K(a+1)$, as desired.
\end{proof}

\subsection*{O-minimality} In this paper, as in \cite{Wilkie}, we work only with a particular o-minimal structure-- $\Ranexp$. A rich class of sets and functions are definable in $\Ranexp$, and this structure appears ubiquitously in applications of o-minimality to diophantine geometry. See \cite{vdDMM} for a description of sets definable in $\Ranexp$.

We shall use an immediate consequence of the defining axioms of o-minimality;  see \cite[Appendix A]{BD}.

\begin{fact}\label{basico}
Let $f, g:\R \to \R$ be definable in $\Ranexp$. Suppose $f(a)=g(a)$ for all sufficiently large $a$, then $f(x)=g(x)$ for all sufficiently large $x$. 
\end{fact}

\noindent
This immediately gives the following improvement to Lemma~\ref{inter}.

\begin{corollary}\label{intero}
Let $f:\R\to\R$ be definable in $\Ranexp$. Suppose there is $K$ such that all sufficiently large $a$, we have that $f(a)\in \Z$ and $$\Delta^{n-a}(\Delta-1)^{a} f(a)=0$$ for all $n$ satisfying $Ka\le n\le K(a+1)$. Then there are $P_1, P_2 \in \Q[X]$ such that $f(x) = P_1(x)  + P_2(x)2^x$ for all sufficiently large $x$.  \end{corollary} 


\noindent
Our next result is a version of \cite[Corollary 4.8]{Wilkie}, and is a key ingredient which allows access to the complex-analytic realm; see also \cite[Theorem 3.1]{GQ}. Set
$$\Rsubexp\coloneq\{f:\R\to \R \text{ definable in $\Ranexp$}: \text{$\forall \epsilon$ $\exists r$ such that $|f(x)|\le e^{\epsilon x}$, $\forall x\ge r$} \}.$$

\noindent
Recall that $\H(B)\subset \C$ denotes the right-half plane comprising complex numbers with real part greater than $B$.
\begin{proposition}\label{3.1}
Let $f:\R\to\R$ be definable in $\Ranexp$ and suppose there are $c>0$ and $C>1$ such that $$|f(x)|\le c\cdot C^x$$
for all sufficiently large $x$. Then for any positive reals $\epsilon$ and $R$, there is $B$ and an analytic function $g: \H(B)\to \C$ such that
\begin{enumerate}
    \item[$\rm(i)$] $|f(x)-g(x)|< e^{-Rx}$ for all $x > B$,
   \item[$\rm(ii)$]  $|g(z)|\le  |C^{z+\epsilon|z|}|$ for all $z\in \H(B)$.
\end{enumerate}
\end{proposition}
\begin{proof}
Let $f$ be as above, and let $\epsilon, R\in \R^{>0}$ be given. Then 
\cite[Corollary 4.8]{Wilkie} gives $N$, $s_1, \ldots, s_N\in \R$, $f_1,\ldots,f_N\in \Rsubexp$, and $B_0$ such that, for all $x> B_0$, $$\left|f(x) - \sum_{j=1}^N f_j(x)\cdot C^{s_jx}\right|<\exp(-Rx).$$ 
We can and do assume that the $s_j$ are distinct, and then the growth condition on $f$ implies that $s_j\le 1$ for all $j=1,\ldots, N$. For each $j\in \{1,\ldots, N\}$, set $g_j(x)\coloneq f_j(x)\cdot C^{s_jx}$. Since $f_j\in \Rsubexp$,  \cite[Lemma  5.3]{Wilkie} gives $B_j$ and a complex analytic continuation $\hat{f}_j:  \H(B_j)\to \C$ of $f_j$ that satisfies $|\hat{f}_j(z)|\le C^{\epsilon|z|}/N$ for all $z\in \H(B_j)$. By the identity theorem, the function $\hg:\H(B_j)\rightarrow \C $ given by $$\hg_j(z)\ \coloneq\ \hat{f}_j(z)C^{s_jz}$$ is an analytic continuation of $g_j$ and $|\hg_j(z)|\le |C^{z+\epsilon|z|}|/N$. 

Observe that $B\coloneq\max\{B_0, B_1,\ldots, B_N\}$ and $g\coloneq\hg_1+\ldots+\hg_N$ are as desired.
\end{proof}

\section{Integer-valued functions definable in $\Ranexp$}\label{IV}

\noindent
We begin this section by revisiting the function $G(k, x, y)$ from \cite[Lemma 6]{Langley}.

\begin{lemma}\label{G}
For fixed $a$, and $k = 0, \ldots, a,$ set
$$G(k, x, y) = \prod_{0 \le q \le k - 1}(1 + x + qy)\prod_{k \le q \le a - 1}(1 - qy),$$
with the convention that a product over an empty range of $q$ is 1.
Let $A_{\mu, \nu}$ denote the coefficients of $\Delta^aG(0, x, y)$, so we have that
$$\Delta^aG(0, x, y)\ =\ \sum_{\substack{\mu \le a,\\ \nu <a}}A_{\mu, \nu}x^\mu y^\nu$$
Then for all $\mu\le a$ and $\nu<a$, $A_{\mu,\nu}=0$ if $\mu +2\nu<a$, and
$|A_{\mu,\nu}|\le 6^a a^{\nu}$.
\end{lemma}
\begin{proof} For each $k\in \{0,\ldots,a\}$, $G(k,x,y)$ is a polynomial 
$$G(k,x,y)=\sum_{\substack{\mu\le a,\\\nu <a}} B_{\mu,\nu}(k) x^{\mu}y^{\nu}.$$
Fix some $\mu\le a$ and $\nu<a-1$. Then \cite[Lemma 6 (ii)]{Langley} implies that $A_{\mu,\nu}=0$ if $\mu+2\nu<a$.

Towards the second assertion, note that  $|B_{\mu,\nu}(k)|\le 3^a a^{\nu}$ for all $k\le a$.
Then the formula (\ref{eq: Delta^n}) gives
$$\Delta^aG(0, x, y) = \sum_{k = 0}^{a}(-1)^{a - k}\frac{a!}{k!(a-k)!}G(k, x, y).$$
It follows that $|A_{\mu, \nu}|\le \sum_{k = 0}^{a} \frac{a!}{k!(a-k)!} |B_{\mu,\nu}(k)|$, and hence $|A_{\mu, \nu}|\le 6^a a^{\nu}$. 
\end{proof}

\noindent
Next, we have a detailed reworking of \cite[Lemma 5]{Langley}.

\begin{lemma}\label{integral}
For $n$ and any positive real $s<n$, let $\mathscr{C}_{n,s}$ be the contour consisting of the arc $ \Omega_{n,s}$ of the circle $|z|=2n$ from $-s-i\sqrt{4n^2-s^2}$ to $-s+i\sqrt{4n^2-s^2}$ described once counter-clockwise, followed by the line segment $T_{n,s}$ from $-s+i\sqrt{4n^2-s^2}$ to $-s-i\sqrt{4n^2-s^2}$. Then there are positive integers $b,d$ such that for all $n\ge 4$, $ 2\le s\le n/2$, and $\mu\le s$ we have that
$$I_{n,s} = \int_{\Omega_{n,s}} \frac{n! 2^{|z|}}{|z|\ldots|z-n|} \left|\frac{z-2n}{n}\right|^{\mu} |dz|\quad <\quad d\cdot (bs/n)^{\mu/2},$$ 
$$J_{n,s} = \int_{T_{n,s}} \frac{n!} {|z|\ldots|z-n|} \left|\frac{z- 2n}{n}\right|^{\mu} |dz|\quad < \quad d\cdot(bs/n)^{s-1} .$$
In fact, the choice of $b=d=100$ suffices.
\end{lemma}
\begin{proof}
Fix $n\ge 4$ and $s$ with $2\le s\le n/2$. We follow the proof of  \cite[Lemma 5]{Langley} with the aim of getting more explicit estimates.

First consider $I_{n,s}$. The arc  $\Omega_{n,s}$ is parameterized by $z=2ne^{i \theta}$, $-d_n\le \theta \le d_n$, where $d_n<2$ since $s\le n/2$. So $1-\cos(\theta)\ge \frac{1}{3}\theta^2$ for all $\theta$ satisfying $-d_n\le \theta \le d_n$. This gives for all $z\in \Omega_{n,s}$ and $0\le k\le n$  that
$$|z-k|^2\ =\ (2n-k)^2 + 4nk(1-\cos\theta)\ \ge\ (2n-k)^2 \exp\left(\frac{nk\theta^2}{3(2n-k)^2}\right).$$
It follows for all $z\in \Omega_{n,s}$ that
\begin{equation}\label{2}
\prod_{k=0}^n|z-k|\ \ge\ \exp(n\theta^2/36)\cdot\prod_{k=0}^n(2n-k).
\end{equation}
Also, we have for $z\in \Omega_{n,s}$ that
\begin{equation}\label{3}
|z-2n|\ \le 2n|\theta|, \qquad \text{and}\qquad \left|\frac{z-2n}{n}\right|^{\mu}\le 2^{\mu}|\theta|^{\mu}.  
\end{equation}
Using (\ref{2}), (\ref{3}), Stirling's approximation, and a substitution $x=\sqrt{n}\theta/6$, we get that
\begin{align*}
I_{n,s}\ & \le \ 2^{\mu} \cdot \frac{(2n)n!2^{2n}}{2n(2n-1)\ldots(2n-n)}\cdot2\int_{0}^{\pi} \theta^{\mu} \exp(-n\theta^2/36)\ d\theta\\
& \le \ 2\cdot 2^{\mu}\cdot4\sqrt{n}\cdot (36/n)^{\mu/2+1/2} \cdot \int_{0}^{\infty} x^{\mu} \exp(-x^2)\ dx\\
& \le \ 48\cdot  2^{\mu}\cdot (36/n)^{\mu/2}\cdot (\mu/2)^{\mu/2}\\
& \le \ 48\cdot (72\mu/n)^{\mu/2}.
\end{align*}
Here we used that $\int_{0}^{\infty} x^{\mu} \exp(-x^2)\ dx=\Gamma(\mu/2+1)/2\le (\mu/2)^{\mu/2}$.

This finishes the proof for the first part of the lemma and we turn our attention to $J_{n,s}$. For $z\in T_{n,s}$, we have $z=-s+iy$ for $-\sqrt{4n^2-s^2}\le y\le \sqrt{4n^2-s^2}$, and 
$$ |z-k|^2= (s+k)^2+y^2$$
for  $0\le k\le n$. It follows that
$$J_{n,s}\ \le \ 2\cdot4^{\mu}\cdot\int_{0}^{\infty} \frac{n!}{\prod_{k=0}^n \sqrt{(s+k)^2 +y^2}}\ dy.$$
Next we observe that for $D$ in \cite[Equation 24]{Langley}, we have that $D\le 2s$ and this gives
\begin{align*}
\frac{n!}{\prod_{k=0}^n \sqrt{(s+k)^2 +y^2}}\ & \le \ \frac{3n^{n+1/2}e^{-n}(s^2+y^2)^{(s-1)/2}e^{n}e^{s}}{\big((n+s)^2+y^2\big)^{(n+s)/2}\exp(y\arctan\big((n+s)/y))} \\
& \le \ 3^{s+1}\cdot n^{n+1/2}n^{-(n+s)}\cdot(s^2+y^2)^{(s-1)/2}e^{-y/4},
\end{align*}
using the simple inequality $\arctan(u)\ge 1/4$ for all $u\ge 1/2$. Now  $\mu\le s$ and $n\ge 4$ give us that
$$J_{n,s}\ \le \ 2\cdot4^s\cdot 3^{s+1}\cdot (n^{-s+1}/2)\cdot \int_0^{\infty} (s^2+y^2)^{(s-1)/2}e^{-y/4}\ dy. $$
Splitting the integral gives
\begin{align*}
\int_0^{\infty} (s^2+y^2)^{(s-1)/2}e^{-y/4}\ dy\ & \le\ 2^{s-1}\left[\int_{0}^{\infty} s^{(s-1)}e^{-y/4}\ dy + \int_{0}^{\infty} y^{(s-1)}e^{-y/4}\ dy\right]\\
    & \le\ 2^{s-1} [4\cdot s^{s-1}\ +\ (4s)^{s-1} ]
\end{align*}
Finally using $s\ge 2$ we arrive at 
$$ J_{n,s}\ \le \ 3^2\cdot 4\cdot 2\cdot (96s/n)^{s-1}. $$
The proof is complete.
\end{proof}

\subsection*{Fixing $K$} Throughout the rest of this section we set $K\coloneq10^5$. Hence $K=10^3b$,
where $b$ is the positive integer given by Lemma~\ref{integral}.

\begin{proposition}\label{bounds}
Let $g:\H(- 1)\to \C$ be analytic and suppose we have  for all $z\in \H(-1)$ that
$$|g(z)|\le |2^{z+2|z|/K}|.$$
Then there is  $C_{\star}<1$ such that for all sufficiently large $a$, we have that
$$|\Delta^{n-a}(\Delta-1)^a g(a)|\ <\ C_{\star}^{a}$$
for all $n$ satisfying  $Ka \le n\le K(a+1)$.
\end{proposition}
\begin{proof} 
Suppose we are given an analytic function $g:\H(-1)\to \C$ such that
$$|g(z)|\le |2^{z+2|z|/K}|$$
for all $z\in \H(-1)$. 
For all positive integers $a$, we consider the translations $g_a$ of $g$ given by complex analytic functions
$$g_a(z)\coloneq g(a+z)\ :\ \H(-a-1)\to \C.$$
So we have for all $a$ that $|g_a(z)|\le 2^{(1+2/K)a}\cdot |2^{z+2|z|/K}|$, and that $g_a$ is analytic on an open set containing the contour $\mathscr{C}_{n,a}$ from Lemma~\ref{integral}. Fix an $a\ge 4$, and an $n$ such that  $Ka \le n\le K(a+1)$.
We will obtain small asymptotic upper bounds for 
$$D_{a,n}\ =\ |\Delta^{n-a}(\Delta-1)^a g(a)|\ =\ |\Delta^{n-a}(\Delta-1)^a g_a(0)|.$$
Following \cite[Section 4]{Langley}, equation (31) there together with Lemma~\ref{G} gives
$$D_{a,n}\cdot\prod_{j=0}^{a-1}(1-j/n)\ \le\ \sum_{\substack{\mu\le a,\\ \nu<a}} |A_{\mu,\nu}| n^{-\nu} \int_{\mathscr{C}_{n,a}} \frac{n!|g_a(z)|}{|z|\ldots|z-n|} \left|\frac{z-2n}{n}\right|^{\mu} |dz|.$$
So we have for all $\mu\le a$ and $\nu<a$  that $$ A_{\mu,\nu}=0\quad  \text{if}\quad \mu +2\nu<a \qquad \text{and}\qquad |A_{\mu,\nu}|\le 6^a a^{\nu}.$$ We have that $|z|\le 2n$ for all $z\in \mathscr{C}_{n,a}$. Now using $a\ge 4$ and $n\le K(a+1)$, we note that $|g_a(z)|\le 2^{7a}2^{|z|} $ for all $z\in \Omega_{n,a}$, and $|g_a(z)|\le 2^{6a} $ for all $z\in T_{n,a}$. Combining these observations with an application of Lemma~\ref{integral} for $s=a$ we get 
\begin{align*}
D_{a,n}\cdot(1-a/n)^{a}\ & \le \ 2^{7a}6^{a}\sum_{\substack{\mu\le a,\\ \nu<a,\\\mu+2\nu\ge a}}  (a/n)^{\nu}\cdot( I_{n,a}+J_{n,a}) \\
& \le\  768^a d \sum_{\substack{\mu\le a,\\ \nu<a,\\\mu+2\nu\ge a}}  (a/n)^{\nu} (ba/n)^{\mu/2} + (a/n)^{\nu} (ba/n)^{a-1},
\end{align*}
where $b, d$ are the constants from Lemma~\ref{integral}. Now using $a/n\le 1/K <1$, we get
\begin{align*}
D_{a,n}\ & \le\ (1-1/K)^{-a}\cdot 768^a\cdot d \cdot a^2 \cdot (b/K)^{a-1}\\
 &\le \ c\cdot \big(800b/(K-1)\big)^{a-1},
\end{align*}
for a constant $c$. Since  $K=1000b$, we get for all $a\ge 4$ that
$$D_{a,n}\ \le\ c\cdot (6/5)^{-a+1}$$
for all $n$ satisfying $Ka\le n\le K(a+1)$, and the proof is complete.
\end{proof}

\noindent
With modest modifications, one can obtain variants of Lemmas~\ref{integral} and \ref{bounds} which can be used together with Lemma~\ref{inter}  to slightly liberalize the growth condition in \cite[Theorem 2]{Langley} (and hence also in Theorem~\ref{P}). We do not pursue this relatively minor increment here, and record it merely as an observation as follows.

\begin{remark}\label{bL}
There is an $\epsilon_0$ such that the following holds. Let $g:\H(r)\to \C$ be analytic for some $r$, and that $g(a)\in \Z$ for all positive integers $a$. Suppose we have  for all $z\in \H(r)$ that
$|g(z)|= O(2^{(1+\epsilon_0)|z|})$. Then there are polynomials $P_1, P_2$ such that $g(z)=P_1(z)+P_2(z)2^z$ for all $z\in \H(r)$. 
\end{remark}

\noindent
We are ready to finish the proof of our first theorem.

\begin{proof}[Proof of Theorem \ref{main}]
Let $f:\R\to \R $ be as in the statement of the theorem, and let $c, B_0$ be such that $f(a)\in \Z$ and $f(x)\le c2^{(1+10^{-5})x}= c2^{(1+1/K)x}$ for all $a\ge B_0$. Apply Proposition~\ref{3.1} for $C=2^{(1+1/K)}$, $\epsilon=1/K$, and $R=K$ to obtain $B_1\ge B_0$ and an analytic function $g:\H(B_1)\to \C$ with \begin{enumerate}
\item[$\rm(i)$] $|f(x)-g(x)|< e^{-Kx}$ for all $x > B_1$,
\item[$\rm(ii)$]   $|g(z)|\le |2^{z+2|z|/K}|$ for all $z\in \H(B_1)$.
\end{enumerate}
Then applying Proposition~\ref{bounds} to an appropriate translate of $g$, we get $B_2> B_1$ such that for all $a\ge B_2$,
$$|\Delta^{n-a}(\Delta-1)^a g(a)|<1/2$$
for all $n$ satisfying  $Ka \le n\le K(a+1)$.
Then by the linearity of $\Delta$ and Lemma~\ref{error}, we obtain $B_3\ge B_2$ such that for all $a\ge B_3$ and $n$ satisfying $Ka\le n\le K(a+1)$ we have that
$$|\Delta^{n-a}(\Delta-1)^a f(a)|\ <\  1, \qquad \text{and hence} \qquad |\Delta^{n-a}(\Delta-1)^a f(a)|=0.$$
We use here the simple observation that since $f(a)\in\Z$ for all $a \ge B_0$, also $\Delta^{n-a}(\Delta-1)^a f(a)\in \Z$ for all $a \ge B_0$. The theorem now follows immediately by Corollary~\ref{intero}.
\end{proof}

\subsection*{Connection to a conjecture of Wilkie} 
We want to draw attention to the following broad conjecture made at the conclusion of \cite{Wilkie}.

\begin{conj}\label{conj: Wilkie}
Let $f:\R\rightarrow \R$ be definable in $\Ranexp$. Suppose that $f(a)\in \Z$ for all sufficiently large positive integers $a$ and there is $r>0$ such that $|f(x)| < \exp(rx)$ for all sufficiently large $x$. Then there is a polynomial $P(x,y_1,\ldots, y_m)$ with rational coefficients, and positive real algebraic integers $\alpha_1,\ldots,\alpha_m$ such that $$f(x) = P(x,\alpha_1^x,\ldots,\alpha_m^x)$$
for all sufficiently large $x$.
\end{conj}

\noindent
Wilkie's result \cite[Theorem 5.1]{Wilkie} settles the case of $r<\ln(2)$, and Theorem~\ref{main} upgrades the classification to all $r\le \ln(2)(1+10^{-5})$. Here we remark also that our choice of $K$ earlier in the section is not optimal, and hence micro-improvements in this exponent $\ln(2)(1+10^{-5})$ may be easily achieved.

\section{Concordant integer-valued functions definable in $\Ranexp$}\label{CIV}

\noindent
We begin with a basic Corollary~\ref{intero} type result.

\begin{fact}\label{Deltao}
Let $f:\R\to\R$ be given. 
\begin{itemize}
\item[$\rm(i)$] Suppose there is $B$ such that $\Delta^nf(B)=0$ for all sufficiently large $n$. Then there is a polynomial $P$ such that $f(a)=P(a)$ for all $a\ge B$. 
\item[$\rm(ii)$] If moreover, $f(a)\in \Z$ for all $a\ge B$, then we get that $P\in \Q[X]$.
\item[$\rm(iii)$] If furthermore, $f$ is definable in $\Ranexp$, then $f(x)=P(x)$ for all large $x$.
\end{itemize}
\end{fact}
\begin{proof}
Item $\rm(i)$ is a classical fact observed by P\'olya, and we sketch a  proof as follows. Let $N$ be such that $\Delta^nf(B)=0$ for all $n\ge N$. Let $P$ be a polynomial of degree at most $N-1$ such that $f(a)=P(a)$ for all $B\le m\le B+N-1$. Then $\Delta^NP(B)=0$, and using $\Delta^Nf(B)=0$, we get that $f(B+N)=P(B+N)$. Proceeding in the same manner, we get that $f(a)=P(a)$ for $a\ge B$, as desired.

If moreover $f(a)\in \Z$ for all $a\ge B$, then the polynomial $P$ above has rational coefficients; and $\rm(iii)$ follows by direct application of Fact~\ref{basico}. 
\end{proof}

\noindent
We shall use the following purely combinatorial fact; the statement is perhaps of some independent interest as well.
\begin{lemma}\label{cmain} 
Let $k\ge 1$ and $p$ be given. For all $\ell\in \N$ we have that
$$\sum_{j = 0}^{k}(-1)^{jp}\binom{kp}{jp}(jp)^{\ell}\ \equiv\ 0 \mod p^k,$$
and also for any $i\in \{1,\ldots, p-1\}$ that 
$$\sum_{j = 0}^{k - 1}(-1)^{jp}\binom{kp}{jp+i}(jp+i)^\ell\ \equiv\ 0 \mod p^{k}.$$
\end{lemma}
\begin{proof}
We fix a primitive root of unity $\zeta_p$. By a mod $p$ reduction of the binomial coefficients we readily observe that
\begin{equation}\label{pp}
(1-\zeta_p)^{p-1}\ =\ py\qquad \text{for some}\ y\in \Z[\zeta_p]. 
\end{equation}
For any $a\in \Q(\zeta_p)$, we let $\mathrm{Tr}(a)$ denote the trace of $a$ with respect to the algebraic field extension $\Q(\zeta_p)/\Q$.

Observe for any $M\in \N$ and $t\in \Z$ that
\begin{align*}
\mathrm{Tr}\big(\zeta_p^{t}(1 - \zeta_p)^{M}\big)\ &=\ \sum_{i = 1}^{p - 1} \zeta_p^{it}(1 - \zeta_p^i)^{M} \ =\ \sum_{i = 0}^{p - 1} \zeta_p^{it}(1 - \zeta_p^i)^{M}\\ 
&=\ \sum_{i = 0}^{p - 1}\sum_{\mu = 0}^{M}(-1)^{\mu}\binom{M}{\mu}\zeta_p^{(\mu+t)i} \\
&= \sum_{\mu = 0}^{M}(-1)^\mu\binom{M}{\mu}\sum_{i = 0}^{p - 1}(\zeta_p^{\mu+t})^i.
\end{align*}
For all $\mu \in \N$, we have that
$$\sum_{i = 0}^{p - 1}(\zeta_p^{\mu+t})^i = 0\ \text{ if }\ p \nmid \mu+t \qquad \text{ and }\qquad  \sum_{i = 0}^{p - 1}(\zeta_p^{\mu+t})^i = p\ \text{ if }\ p \mid \mu+t,$$
which gives, for $t > -p,$ that
\begin{equation}\label{Tr}
\mathrm{Tr}\big(\zeta_p^{t}(1 - \zeta_p)^{M}\big)\ =\ p\sum_{j = \max (0, \lceil t/p\rceil)}^{\lfloor (M+t)/p\rfloor}(-1)^{j p-t}\binom{M}{j p-t}.
\end{equation}
Note this implies, by the linearity of the trace operator, that $p\mid \mathrm{Tr}\big((1 - \zeta_p)^{M} y'\big)$ for any $M\in \N^{\ge 1}$ and $y'\in \Z[\zeta_p]$.

We now work out the first assertion. Applying (\ref{Tr}) for $M=kp$ and $t=0$ gives
$$\mathrm{Tr}\big((1 - \zeta_p)^{kp}\big)\ =\ p\sum_{j = 0}^{k}(-1)^{jp}\binom{kp}{jp}.$$
Then the identity (\ref{pp}) implies that 
$$p^k\cdot \mathrm{Tr}\big((1 - \zeta_p)^{k}y^k\big)\ =\ p\sum_{j = 0}^{k}(-1)^{jp}\binom{kp}{jp}.$$
Then by the remark just after (\ref{Tr}),
$$p\mid \mathrm{Tr}\big((1 - \zeta_p)^{k}y^k\big)\quad \text{ and hence}\quad  \sum_{j = 0}^{k}(-1)^{jp}\binom{kp}{jp}\ \equiv\ 0 \mod p^k.$$
Hence we have the first assertion for the case $\ell=0$. 

Note that this claim is immediate for $\ell \ge k$. For any $1\le \ell <k$, consider
$$\sum_{j = 0}^{k}(-1)^{jp}\binom{kp}{jp}(jp)(jp - 1)\ldots(jp - \ell+ 1)= \frac{(kp)!}{(kp - \ell)!}\sum_{j = \lceil \ell/p\rceil}^{k}(-1)^{jp}\binom{kp - \ell}{jp - \ell}.$$
By (\ref{Tr}) for $M=kp-\ell$ and $t=\ell$, and (\ref{pp}) we see that
$$p\sum_{j = \lceil \ell/p\rceil}^{k}(-1)^{jp-\ell}\binom{kp - \ell}{jp - \ell}\ = \ \mathrm{Tr}\big(\zeta_p^{\ell}(1 - \zeta_p)^{kp-\ell}\big)\ =\ p^k\cdot \mathrm{Tr}\big(\zeta_p^{\ell}(1 - \zeta_p)^{k-\ell}y^k\big). $$
Then as before we use the remark just after (\ref{Tr})  to conclude
$$ \sum_{j = 0}^{k}(-1)^{jp}\binom{kp}{jp}(jp)(jp - 1)\ldots(jp - \ell+ 1)\ \equiv\ 0 \mod p^k. $$
From the above equation, the first assertion follows for $\ell = 1$. For $\ell >1$, note we have integers $b_1, \ldots, b_{\ell - 1}$ such that 
$$jp(jp - 1)\ldots(jp - \ell +1)\ =\ (jp)^{\ell} + b_1(jp)^{\ell - 1} + \ldots + b_{l - 1}(jp),$$ 
and hence a simple induction on $\ell$ completes the proof of the first assertion.

We follow a similar process for the proof of the other claim as well, and consider first the case of $\ell=0$. Fix $i\in \{1,\ldots, p-1\}$. Then applying (\ref{Tr}) for $M=kp$ and $t=-i$, we get that
$$\mathrm{Tr}\big(\zeta_p^{-i}(1 - \zeta_p)^{kp}\big)\ =\ p\sum_{j = 0}^{k - 1}(-1)^{i + jp}\binom{kp}{ jp +i}.$$
Then the same arguments as before imply the desired result. We leave the general $\ell\ge  1$ case of the second assertion as an exercise to the reader.
\end{proof}

\noindent
This lemma will be used to obtain a strong divisibility implication of concordance. Recall the following definition from the introduction.

\begin{definition} A function $f:\R\to \R$ is $k$-\textit{concordant} if for each tuple $m_0,\dots,m_{k}$  of sufficiently large positive integers, there is a polynomial $P\in\Z[X]$ such that $P(m_i)=f(m_i)$ for all $i=0,\ldots, k$. Note we do not require the $m_0,\ldots, m_k$ in the statement above to be distinct.
\end{definition}

\begin{remark}\label{remcon} Clearly, $f:\R\to \R$ is  $0$-concordant if and only if $f(a)\in \Z$ for all sufficiently large $a$. Note that $1$-concordance is equivalent to the condition that $f(m_1)\equiv f(m_2)\bmod(m_1-m_2)$ for all sufficiently large  $m_1, m_2$. For all $k$, it follows directly from definition that $k+1$-concordance implies $k$-concordance, and hence $k$-concordant functions take integer values for all sufficiently positive integer inputs. 
\end{remark}

\begin{remark}\label{PeZa}
In \cite{PeZ}, Perelli and Zannier consider functions satisfying a weaker form of $1$-concordance. For an integer-valued entire function $f:\C\to \C$, they show that if  $f(a+p) \equiv f(a) \bmod p$ for all positive integers $a$ and large enough primes $p$, then  $|f(z)|=O(C^{|z|})$, for some $C<e+1$, implies that $f$ must be a polynomial. We remark that the conclusions of the following lemma and Theorem~\ref{concord} for the case of $k=1$ hold even when the function $f:\R\to \R$ satisfies only this  weaker form of 1-concordance. The details are left to the reader. 
\end{remark}

\noindent
Recall $\gamma_k=\sum_{j=1}^k 1/j$. The following corresponds to \cite[Theorem 3.5, Lemma 5.1]{Pila}.

\begin{lemma}\label{concordant}
Let $f:\R \to \R$ be $k$-concordant, with $k\ge 1$. Then there are $c, B$ such that for all $a\ge B$ and all $n$ we have that  $$|\Delta^nf(a)|\ \ge\ e^{\gamma_k n-c n/\log(n)}\qquad \text{ or }\qquad \Delta^nf(a)\ =\ 0.$$
\end{lemma}
\begin{proof}
Let a $k$-concordant $f:\R \to \R$ be given. Fix a $B$ such that for all $m_0,\ldots, m_k \in \N^{\ge B}$, there is a polynomial $P\in \Z[X]$  with $P(m_i)=f(m_i)$ for all $i=0,\ldots,k$. We claim that for any $a\ge B$ and prime $p$ that
\begin{equation}\label{claim}
\Delta^{kp}f(a) \equiv 0 \mod p^k.
\end{equation}
To show this, we fix a $p$ and use (\ref{eq: Delta^n}) to get that
$$\Delta^{kp}f(a) = (-1)^{kp}\sum_{i = 0}^{kp}(-1)^{i}\binom{kp}{i}f(a + i).$$
This can be rearranged as follows
$$(-1)^{kp}\Delta^{kp}f(a)\ =\ \sum_{j = 0}^{k}(-1)^{jp}\binom{kp}{jp}f(a + jp)\ +\ \sum_{i = 1}^{p - 1}\sum_{j = 0}^{k - 1}(-1)^{jp+i}\binom{kp}{jp+i}f(a + jp+i).$$
Consider the first sum on the right hand side. Since $a\ge B$, the $k$-concordance of $f$ gives a polynomial $P\in\Z[X]$ such that $P(jp)=f(a+jp)$ for all $j=0, \ldots, k$.
Now the first assertion of Lemma~\ref{cmain} implies that
$$\sum_{j = 0}^{k}(-1)^{jp}\binom{kp}{jp}f(a + jp)\ =\ \sum_{j = 0}^{k}(-1)^{jp}\binom{kp}{jp}P(jp)\  \equiv\ 0 \mod p^k.$$
A similar argument, now involving the second assertion of Lemma~\ref{cmain}, gives for each $i\in \{1,\ldots, p-1\}$ that
$$\sum_{j = 0}^{k - 1}(-1)^{jp+i}\binom{kp}{jp+i}f(a + jp+i)\  \equiv\ 0 \mod p^k,$$
and we have proved the claim set forth in (\ref{claim}).

So we have for any $p\le \frac{n}{k}$  that
$\Delta^nf(a) \equiv 0 \mod p^k$, and hence we get for all $\ell=1,2,\ldots,k,$ $$\prod\limits_{ p \le \frac{n}{\ell}} p^{\ell}\ \mid \ \Delta^nf(a),$$
with an empty product being $1$ by convention. In other words,
$$\prod\limits_{\ell \le k}\prod\limits_{p \le \frac{n}{\ell}} p\ \mid \ \Delta^nf(a).$$
The prime number theorem gives a constant $c$ such that
$$\sum_{\ell \le k}\sum_{p\le \frac{n}{l}}\log (p)\ \ge\ \gamma_k n - c n/\log(n)  $$
for all $n$.  Thus we get for all $n$ that either $$|\Delta^nf(a)|\ \ge\ e^{\gamma_kn-cn/\log(n)}\qquad \text{ or }\qquad \Delta^nf(a)=0,$$
and the proof is complete.
\end{proof}

\begin{lemma}\label{bounds1}
Let analytic $g:\H(-2)\to \C$ and $k\ge 1$ be given. Suppose there are $\epsilon<1$, and $C$ such that
$$|g(z)|\le |C^{z+\epsilon|z|}|$$
for all $z\in \H(-2)$. If $C^{1+\epsilon}<e^{\gamma_k}+1$, then there is $C_\star<e^{\gamma_k}$ such that for all sufficiently large $n$ we have that
$$|\Delta^{n} g(0)|\ \le\  C_{\star}^n.$$ 
\end{lemma}
\begin{proof}
Let an analytic function $g:\H(-2)\to \C$, and positive integer $k$ be given. Let $\epsilon, C$ be such that $|g(z)|\le |C^{z+\epsilon|z|}|$ for all $z\in \H(-2)$, and also $C^{1+\epsilon}<e^{\gamma_k}+1$. If $C\le 1$, then $|\Delta^{n} g(0)|\le 2^n$ and we are done. Hence we assume that $C>1$.

Fix an $n$, and let $\mathscr{C}_{n}$ be the contour formed by traversing once counter-clockwise the boundary of the region inside $\{z\in\C:\ \mathrm{Re}(z)\ge -1\}$ that is bounded by the circle $|z|=(1+1/e^{\gamma_k})n$  and the vertical line $\mathrm{Re}(z)=-1$. We denote by $T_n$ the part of $\mathscr{C}_n$ that is contained in the vertical line, and let $\Omega_n$ denote the rest of the contour, so $\mathscr{C}_n=\Omega_n\cup T_n$. Then Cauchy's integral formula and (\ref{eq: Delta^n}) gives that $$\Delta^ng(0) = \frac{1}{2\pi i} \int_{\mathscr{C}_{n}}\frac{n!g(z)}{z(z-1)\ldots(z-n)}dz.$$
For all $z\in \Omega_n$,  $|g(z)|\le (C^{1+\epsilon})^{(1+1/e^{\gamma_k})n}$, and $|z-k|\ge |z|-k= (1+1/e^{\gamma_k})n-k$ for all $k=0,1,\ldots,n$. A straightforward application of Stirling's approximation gives $c_1$ such that 
$$\frac{n!}{(1+1/e^{\gamma_k})n\big((1+1/e^{\gamma_k})n-1\big)\ldots\big((1+1/e^{\gamma_k})n-n\big)}\ < \ \frac{c_1\cdot e^{\gamma_k n}}{(e^{\gamma_k}+1)^{(1+1/e^{\gamma_k})n}}.$$
Then changing $c_1$ as required, we get for $C_1\coloneq e^{\gamma_k}\cdot  \big(C^{1+\epsilon}/(e^{\gamma_k}+1)\big)^{1+1/e^{\gamma_k}} $ that
$$\int_{\Omega_{n}}\frac{n!|g(z)|}{|z|\cdot|z-1|\cdots|z-n|}dz \ <\ c_1 \cdot n C_1^n;$$
note that $C_1<e^{\gamma_k}$.
We now turn our attention to the integral over $T_n$. For all $z\in T_n$, we have that $|g(z)|\le  (C^{\epsilon})^{(1+1/e^{\gamma_k})n}$, and $|z-k|\ge k+1$ for all $k=0,1,\ldots,n$. 

Using $ C^{1+\epsilon}<e^{\gamma_k}+1$ and $\epsilon<1$ it follows that $C^{2\epsilon}< e^{\gamma_k}+1$. Then since $x^4\ge (x+1)^3$ for all $x\ge e$, we get that $C_2\coloneq(C^{\epsilon})^{(1+1/e^{\gamma_k})}<C^{3\epsilon/2}<e^{\gamma_k}$. So there is $c_2$ such that 
$$ \int_{T_{n}} \frac{n!|g(z)|}{|z|\cdot|z-1|\cdots|z-n|}\ \le  \int_{T_{n}} n^{-1} \cdot C_2^n\ \le\ c_2\cdot C_2^n. $$
Hence it is immediate that the desired conclusion follows for any choice of $C_{\star}$ satisfying $\max(C_1, C_2)<C_{\star}<e^{\gamma_k}$.
\end{proof}

\begin{proof}[Proof of Theorem \ref{concord}]
Let $k$-concordant $\Ranexp$-definable $f:\R\to \R $ be given, and suppose there is $C< e^{\gamma_k}+1$ such that $|f(x)|\le C^x$ for all sufficiently large $x$. The case is $k=0$ follows  by Theorem~\ref{main} (or \cite[Theorem 5.1]{Wilkie}). So we assume $k\ge 1$ and apply Lemma~\ref{concordant} to obtain $c,B_0$ such that
\begin{equation}\label{8}
|\Delta^nf(a)|\ \ge\ e^{\gamma_k n-c n/\log(n)}\qquad \text{ or }\qquad \Delta^nf(a)\ =\ 0
\end{equation}
for all $a\ge B_0$ and all $n$. Fix any $\epsilon<1$ such that $C^{1+\epsilon}<e^{\gamma_k}+1$. Proposition~\ref{3.1} to gives $B\ge B_0$ and an analytic function $g:\H(B-2)\to \C$ with \begin{enumerate}
\item[$\rm(i)$] $|f(x)-g(x)|< 1$ for all $x \ge B$,
\item[$\rm(ii)$]   $|g(z)|\le |C^{z+\epsilon|z|}|$ for all $z\in \H(B-2)$.
\end{enumerate}
This implies the easy estimate $|\Delta^n\big(f(B)-g(B)\big)|\le 2^n$ for  all $n$.  Then the linearity of $\Delta$, and Lemma~\ref{bounds1} applied to an appropriate translate of $g$, gives $C_{\star}<e^{\gamma_k}$  such that for all large sufficiently large $n$ we have that
$$\Delta^n f(B)\ \le \ C_{\star}^n.$$
This together with (\ref{8}) implies that $\Delta^n f(B)=0$ for all sufficiently large $n$, and finally the theorem follows by Fact~\ref{Deltao}.
\end{proof}

\bibliographystyle{amsplain}

\bibliography{BMRW}

\end{document}